\documentclass[10pt]{article}  % 08/12/2021,  
\usepackage{amssymb}
\usepackage{amsmath,amsbsy,amsfonts,amsthm}
\usepackage{graphicx}
\usepackage{geometry}

\usepackage{hyperref}
\usepackage{cleveref}

\newtheorem{theorem}{Theorem}[section]
\newtheorem{lemma}[theorem]{Lemma}
\newtheorem{proposition}[theorem]{Proposition}
\newtheorem{corollary}[theorem]{Corollary}

   \newcommand{\R}{\mathbb{R}}

\newcommand\tr{\mathrm{tr}}
\newcommand\dv{\mathrm{div}}

\begin{document}

\numberwithin{equation}{section}

\title{On Generalized quasi-Einstein Manifolds}
\author{Antonio Airton Freitas Filho \footnote{ Universidade Federal do Amazonas,
 Department of Mathematics, Av. Rodrigo Oct\'avio, 6200, 69080-900 Manaus, Amazonas, Brazil, aafreitasfilho@ufam.edu.br; http://www.ufam.edu.br. Partially supported by CAPES/Brazil-Finance Code 001. 
 } \qquad Keti Tenenblat\footnote{ Universidade de Bras\'{\i}lia,
 Department of Mathematics,
 70910-900, Bras\'{\i}lia-DF, Brazil, K.Tenenblat@mat.unb.br Partially supported by CNPq Proc. 312462/2014-0, Ministry of Science and Technology, Brazil and CAPES/Brazil-Finance Code 001.}
 %  and FAPDF/Brazil grant 0193.001346/2016.
  }

\date{}

\maketitle{}

\begin{abstract}
A rigidity result for a class of compact generalized quasi-Einstein manifolds  with constant scalar curvature is obtained. Moreover, under some geometric assumptions, the rigidity for the noncompact case is also proved. Considering non constant scalar curvature, we characterize the generalized quasi-Einstein manifolds which are conformal to the Euclidean space and we show that there exist two classes of complete manifolds, which  are obtained by considering potential functions and conformal factors either to be radial  or invariant under the action of an (n-1) dimensional translation group. Explicit examples are given. 
\\  \\   
%\noindent 
{\em Keywords: }  Generalized quasi-Einstein Manifolds, Einstein type manifolds, Rigidity results, Conformally flat complete manifolds.   \\   
%\noindent  
{\em Mathematics Subject Classification 2010: } Primary: 53C25, 53C44, 53C24; Secondary: 53C21.
\end{abstract}

\maketitle

\section{Introduction}\label{int} 
The investigation of Einstein manifolds and their generalizations  has been of great interest throughout recent years. In this paper, we study gradient generalized quasi-Einstein manifolds.  This concept was introduced by Catino~\cite{Catino} as follow. 

Let $(M^n,g)$ be a connected Riemannian manifold of dimension $n\geqslant3$. We denote by $Ric$ and $S$ its Ricci tensor and scalar curvature, respectively. We say that $(M^n,g,f,\lambda,\nu)$ is a {\it gradient generalized quasi-Einstein manifold } or $(M^n,g)$ supports a gradient generalized quasi-Einstein structure if there exist smooth real functions $f,\ \lambda$ and $\nu$ defined on $M$, such that
\begin{align}\label{gqEm}
Ric+\nabla^2f-\nu df\otimes df =\lambda g,
\end{align}
where $\nabla^2f$ stands for the Hessian of $f$. We refer to $f$ as the {\it potential function}.  We say that the generalized quasi-Einstein manifold is trivial when $f$ is constant.

In \cite{Catino}, Catino  proved that a generalized quasi-Einstein manifold, with harmonic Weyl tensor and with vanishing radial Weyl curvature, is locally a warped product, with $(n-1)$-dimensional Einstein fibers. 

Moreover, Catino et al.~ \cite{Catinoetal} also introduced the notion of an Einstein-type structure on a Riemannian manifold, unifying various cases studied recently, such as gradient Ricci solitons, Yamabe solitons and quasi-Einstein manifolds  (see \cite{Cao}, \cite{Case}, \cite{Hamilton} and \cite{Pigola}).  A Riemannian manifold $(M^n, g)$ is said to be a {\it gradient Einstein-type manifold}  if there exist $f\in C^\infty(M)$ and constants 
$\alpha,\beta,\mu,\rho \in \mathbb{R} $, with $(\alpha,\beta,\mu)\neq (0,0,0)$, such that  
\begin{align}\label{gradtype}
\alpha Ric+\beta\nabla^2f+\mu df\otimes df&=(\rho S+\lambda)g,
\end{align}
where $\lambda\in  C^\infty(M)$. 

There is an intersection between those two notions. In fact, given a gradient 
generalized quasi-Einstein manifold satisfying \eqref{gqEm}, 
 if $\nu $ is a constant function 
 and $\lambda$ 
is a linear function of the scalar curvature then, it is a gradient Einstein-type manifold. On the other hand, given a gradient Einstein manifold satisfying \eqref{gradtype}, if the constants $\alpha$ and $\beta$ are non zero, then dividing by $\alpha$ and  considering 
$\tilde{f}=\beta f /\alpha$, $\tilde{\nu}=-{\mu\alpha}/{\beta^2}$ and  $\tilde{\lambda}=(\rho S+\lambda)/\alpha$, we have a gradient generalized quasi-Einstein manifold satisfying  $Ric+\nabla^2\tilde{f}-\tilde{\nu}d\tilde{f}\otimes d\tilde{f}=\tilde{\lambda}g$.	

 A Ricci-Hessian type manifold is another example of structure that generalizes an Einstein manifold and it is related to the construction of gradient Ricci (almost) soliton warped product (see \cite{RSWP,ARSWP}). For a complete characterization of gradient Ricci almost solitons on semi-Riemannian warped product see \cite{B-Tenen}.

Under certain conditions  compact Riemannian manifolds of constant scalar curvature are isometric to a standard sphere, for example, Gomes~\cite{Naza} proved that a compact gradient Einstein-type manifold of constant scalar curvature is isometric to the standard sphere with a well defined potential function. The Yamabe problem assures that a compact smooth manifold admits a Riemannian metric of constant scalar curvature \cite{yamabe}.
 Compact Einstein manifolds emerge as critical points of the total scalar curvature functional restricted to the Riemannian metrics of volume $1$ ~\cite{besse}. 

This paper contributes in two directions. The first one provides rigidity  results for complete quasi-Einstein type manifolds with constant scalar curvature $S$. The other one provides classes of complete quasi-Einstein type manifolds, conformal to the Euclidean space, with non constant curvature $S$.

We obtain rigidity results for compact (and noncompact) generalized quasi-Ein\-stein type manifolds with  constant scalar curvature, similar to those obtained in \cite{Naza} for the generalized Einstein-type manifolds.   assume an additional condition,  which is  motivated by a necessary property  
satisfied by any generalized quasi-Einstein structure (see Proposition \ref{gen} ), namely: 
\begin{align}\label{3form}
d\|\nabla f\|^2\wedge d\nu\wedge df=0, 
\end{align} 
where $f$ and $\nu$ are the functions satisfying \eqref{gqEm}.
 
 Considering \eqref{3form}, if $\nu$ is not a constant function, we  assume that $\nu$ is a function 
of $f$, i.e. , $\nu=v\circ f$ for some smooth function $v:I\subset\mathbb{R}\to\mathbb{R},$ where $I=f(M)$.  
This same assumption was considered in \cite{Mir-Beh}, where Mirshafeazadeh-Bidabad  showed that on a Bach-flat compact generalized quasi-Einstein manifold $(M^n,g,f,\lambda,\nu)$, when the function $\nu$ is a positive function of the potential $f$, then  certain tensors vanish.

%%%?? VER VER nomes

In order to state  our rigidity results, Theorems \ref{R1} and \ref{R2}, 
assuming that $\nu=v\circ f$,  we define the following smooth function $\phi:I\subset\mathbb{R}\to\mathbb{R}$ given by
\begin{align}\label{OD1}
\phi(t)=c_1\int^t\left(e^{-\int^s v(r)dr}\right)ds+c_2,
\end{align}
for some constants $c_1$ and $c_2,$. 

Our first rigidity result is a characterization of the standard sphere.

\begin{theorem}\label{R1}
Let $(M^n,g,f,\lambda,\nu)$ be a compact generalized quasi-Einstein manifold of constant scalar curvature $S$ and $\nu=v\circ f$ for some smooth function $v:I\subset\mathbb{R}\to\mathbb{R},$ where $I=f(M).$ Then, $(M^n,g)$ is isometric to the standard sphere $\mathbb{S}^n(r)$ of radius $r$.  Moreover, up to rescaling, the potential function is given by $f=\phi^{-1}\Big(c-h_w /n \Big)$, where   $\phi$ is given by \eqref{OD1}, $c$ is a constant and $h_w$ is the height function on the unit sphere $\mathbb{S}^n$, with respect to a unit vector $w$.     
\end{theorem}

For the  noncompact complete  case, we consider a geometric assumption motivated by Karp's theorem (see Section \ref{pre}). More precisely 

\begin{theorem}\label{R2}
Let $(M^n,g,f,\lambda,\nu)$ be a complete noncompact generalized quasi-Einstein manifold of constant scalar curvature $S$ and $\nu=v\circ f$ for some smooth function $v:I\subset\mathbb{R}\to\mathbb{R},$ where $I=f(M).$ Consider the geodesic ball $B(r)$ of radius $r$ centered at some fixed point $x_\circ\in M.$ In addition,  suppose that 
\begin{align}\label{T1}
\displaystyle{\liminf_{r\to\infty}} \frac{1}{r}\int\limits_{B(2r)\backslash B(r)}\|\mathring{Ric}(\nabla(\phi\circ f))\|dvol_g=0,
\end{align}
where $\phi:I\subset\mathbb{R}\to\mathbb{R}$ is given by \eqref{OD1} and $\mathring{Ric}$ denotes the traceless Ricci tensor. Then, $(M^n,g)$ is an Einstein manifold of non positive scalar curvature $S$ and $f$ has at most one stationary point. More precisely, $(M^n,g)$ is isometric to one the following manifolds: 
\begin{itemize}
\item[$(i)$] a Riemannian Product $\mathbb{R}\times N^{n-1}.$ 
\item[$(ii)$] an Euclidean Space $\mathbb{R}^n.$
\item[$(iii)$] a Hyperbolic Space $\mathbb{H}^n\left(\sqrt{\frac{n(n-1)}{|S|}}\right).$ 
\item[$(iv)$] a Warped Product $\left(\mathbb{R}\times N^{n-1},dt^2+e^{2\sqrt{\frac{|S|}{n(n-1)}}t}g_N\right),$ where $(N^{n-1},g_N)$ is a Riemannian Einstein manifold.
\end{itemize}
\end{theorem}

In a recent paper, \cite{GWX} Gomes-Wang-Xia considered the $h$--almost Ricci solitons. This concept was originally introduced in \cite{maschler} to study  conformal changes of K\"ahler-Ricci solitons which give rise to new K\"ahler metrics. In \cite{GWX} (see Theorem 1), The authors
proved that a compact $h$--almost Ricci soliton,  with constant scalar curvature and non vanishing $h$ function, is isometric to a standard sphere. 
 With this result  in mind, we  recall that if $(M^n,g),$ $n\geqslant3,$ is a compact Riemannian manifold with constant scalar curvature, then it does not admit an Einstein metric $\overline{g}$ conformal to $g$, unless $(M^n,g)$ is  isometric to a standard sphere. This  is  a well known characterization of the standard sphere obtained by Obata (\cite{obata2} Proposition 6.2). Assuming hypothesis similar to those of Theorem \ref{R2},  we obtain the following interesting  rigidity  result, for the complete noncompact case, namely:

\begin{corollary}\label{aaa}
Let $(M^n,g)$ be a complete noncompact Riemannian manifold with constant scalar curvature $S$. Assume that $\overline{g}=\frac{1}{\varphi^2}g$ is an Einstein metric on $M^n$ where $\varphi$ is a positive smooth function. Let  $B(r)$ be the geodesic ball of radius $r$ centered at some fixed point $x_\circ\in M$.  In addition suppose that	
\begin{align}\label{Ta}
\displaystyle{\liminf_{r\to\infty}} \frac{1}{r}\int\limits_{B(2r)\backslash B(r)}\|\mathring{Ric}(\nabla\varphi)\|dvol_g=0.
\end{align}
Then, $(M^n,g)$ is an Einstein manifold with  $S\leq 0$ and $\varphi$ has at most one stationary point. More precisely, $(M^n,g)$ is isometric to one the following manifolds: 
\begin{itemize}
\item[$(i)$] a Riemannian Product $\mathbb{R}\times N^{n-1}.$ 
\item[$(ii)$] the Euclidean Space $\mathbb{R}^n.$
\item[$(iii)$] the  hyperbolic Space $\mathbb{H}^n\left(\sqrt{\frac{n(n-1)}{|S|}}\right).$ 
\item[$(iv)$] a Warped Product $\left(\mathbb{R}\times N^{n-1},dt^2+e^{2\sqrt{\frac{|S|}{n(n-1)}}t}g_N\right),$ where $(N^{n-1},g_N)$ is a Riemannian Einstein manifold.	
\end{itemize}
\end{corollary}

The rigidity results stated above, require constant scalar curvature. If we release this requirement, then we obtain the following theorems, which prove the existence of classes of complete generalized quasi-Einstein manifolds, conformal to the Euclidean space $(\mathbb{R}^n,g_\circ)$, where $g_\circ$  is the standard metric. These manifolds are obtained by considering the conformal factor and the potential function to be either radial or invariant under the action of an $(n-1)$ dimensional translation group. 

\begin{theorem}\label{radialcomplete}
Let $(\mathbb{R}^n,g)$, $n\geq 3$, be a conformally flat Riemannian manifold, where $g=g_\circ/ \varphi^2$.
Assume that $\varphi(r)$ and $f(r)$ are radial  functions of $r=\|x\|^2$, such that $\varphi$ does not vanish and $f$ is strictly monotone.  Then  $(\mathbb{R}^n,g,f,\lambda,\nu)$ is a  generalized quasi-Einstein manifold if, and only if, $\lambda$ and $\nu$ are also functions of $r$ which are given by 
\begin{align}\label{rad1n}
\nu =\frac{1}{(f')^2}\left[(n-2)\frac{\varphi''}{\varphi}+f''+2f'\frac{\varphi'}{\varphi}\right], 
\end{align}
\begin{align}\label{rad2n}
\lambda=4 \ [ (n-1)\varphi'(\varphi - r\varphi')+r \varphi(\varphi''-f'\varphi ')]+2f'\varphi^2.
\end{align}
In this case,  
\begin{equation}\label{radS}
S=4(n-1)[ 2r\varphi\varphi''-nr(\varphi')^2+ n\varphi\varphi'].
\end{equation}
Moreover,  if $\varphi$ is a bounded function then the manifold is complete.
\end{theorem}

\begin{theorem}\label{translation}
Let $(\mathbb{R}^n,g)$ be a conformally flat Riemannian manifold, where 
$g=g_\circ/ \varphi^2$ and  $\varphi(u)$ is a non vanishing function of $u=\sum_{k=1}^n\alpha_k x_k$, where  $\alpha_k\in\R$ are such that $a=\sum_{k=1}^n \alpha_k^2\neq 0$.  
Let $f(u)$ be a strictly monotone function.  Then there exist  $\nu(u)$ and 
$ \lambda(u)$  such that  $(\mathbb{R}^n,g,f,\lambda,\nu)$ is a gene\-ralized quasi-Einstein manifold if, and only if, 
 \begin{equation}\label{nutransl}
\nu=\frac{1}{(f')^2}\left[(n-2)\frac{\varphi''}{\varphi}+f''+2f'\frac{\varphi'}{\varphi}\right], 
\end{equation}
  \begin{equation}\label{lambdatransl}
\lambda=a \,[\varphi\varphi''-f'\varphi\varphi'-(n-1)(\varphi')^2].
\end{equation}
Its scalar curvature is given by 
\begin{equation}\label{Stransl}
S=a(n-1)[2\varphi\varphi''-n(\varphi')^2].
\end{equation}
 Moreover,  if $\varphi$ is a bounded function then the manifold is complete.
 \end{theorem}

In both theorems above, the generalized quasi-Einstein manifolds, given by Theorems \ref{radialcomplete} and \ref{translation} necessarily  satisfy the property that $\nu$ is a function of $f$, which was required in the rigidity results. This follows from the fact that we are considering  the potential function $f$ to be strictly monotone. In fact, one can obtain $r$ in terms of $f$ and hence also $\lambda$ and $\varphi$ are functions of $f$. 

This paper is organized as follows. In Section \ref{pre}, we obtain some preliminary results that will be used in the proofs of the next section. In Section \ref{proofrigid}, we prove the rigidity results given in   
Theorems \ref{R1}, \ref{R2} and Corollary \ref{aaa}. In Section \ref{nonconstantS}, we prove Theorems \ref{radialcomplete} and \ref{translation} which  
provide  ge\-ne\-ra\-lized quasi-Einstein manifolds, conformal to the $n$--dimensional Euclidean space and we give  
%  with non constant scalar curvature $S$.  
 explicit examples of complete ones.

\section{Preliminaries}\label{pre}
In this section, we present the necessary results that will be used  in
Section \ref{proofrigid}, where we prove our rigidity theorems. We begin with the proposition that motivates one of our hypothesis.

\begin{proposition}\label{gen} 
	Let $( M^n,g,f,\lambda,\nu  ) $ be a generalized quasi-Einstein manifold. Then 
	\begin{align*} % \label{3form}
		d\|\nabla f\|^2\wedge d\nu\wedge df=0.
	\end{align*}
\end{proposition}

\begin{proof}
	Since $\dv Ric=\frac{1}{2}dS,$ we compute both sides of this equation.  
		\begin{align*}
		\dv Ric&= -\dv\left(\nabla^2f\right)+\dv\left(\nu df\otimes df\right)+\dv\left(\lambda g\right)\\
		&=-Ric(\nabla f,\cdot)-d\Delta f +\nu\Delta fdf+\frac{\nu}{2}d\|\nabla f\|^2+\langle\nabla f,\nabla\nu\rangle df + d\lambda,
	\end{align*}
	and since
	\begin{align*}
		Ric(\nabla f,\cdot)=-\frac{1}{2}d\|\nabla f\|^2+\nu\|\nabla f\|^2df+\lambda df,
	\end{align*}
	we have
	\begin{align}
		\dv Ric 
		&= \frac{1}{2}d\|\nabla f\|^2-\nu\|\nabla f\|^2df-\lambda df-d\Delta f  \nonumber\\
		&\hspace{.38cm}+\nu\Delta fdf+\frac{\nu}{2}d\|\nabla f\|^2+\langle \nabla f,\nabla \nu\rangle df+d\lambda \nonumber \\
		&= \frac{1}{2}d\|\nabla f\|^2-d\left(\lambda f\right)+fd\lambda-d\Delta f+d\lambda\nonumber\\
		&\hspace{.38cm}+\dv_f\left(\nu\nabla f\right)df+\frac{\nu}{2}d\|\nabla f\|^2, \label{LHS}
       	\end{align}
where $\dv_f(X):=\dv(X)-\langle X,\nabla f\rangle,$ for all $X\in\mathfrak{X}(M).$
	
	Moreover, it follows from \eqref{gqEm} that 
	\begin{align}
		\frac{1}{2}dS&=\frac{1}{2}d\left(-\Delta f+\nu\|\nabla f\|^2+n\lambda\right)\nonumber\\
		&=-\frac{1}{2}d\Delta f+\frac{\nu}{2}d\|\nabla f\|^2+\frac{\|\nabla f\|^2}{2}d\nu+\frac{n}{2}d\lambda.\label{RHS}
	\end{align}
	
	Therefore, \eqref{LHS} and \eqref{RHS} imply that 
	\begin{align*}
		-\frac{1}{2}d\Delta f+\frac{\nu}{2}d\|\nabla f\|^2+\frac{\|\nabla f\|^2}{2}d\nu+\frac{n}{2}d\lambda&=\frac{1}{2}d\|\nabla f\|^2-d\left(\lambda f\right)+fd\lambda-d\Delta f\\
		&\hspace{.38cm}+d\lambda+\dv_f\left(\nu\nabla f\right)df+\frac{\nu}{2}d\|\nabla f\|^2, 
	\end{align*}
	which reduces to	
	\begin{align*}
		\frac{1}{2}d\Delta f-\frac{1}{2}d\|\nabla f\|^2+\frac{(n-2)}{2}d\lambda+d\left(\lambda f\right)+\frac{\|\nabla f\|^2}{2}d\nu = \dv_f\left(\nu\nabla f\right)df+fd\lambda,
	\end{align*}
i.e. 
		\begin{align*}
		\frac{1}{2}d\left[\Delta_ff+(n-2)\lambda+2\lambda f\right]+\frac{\|\nabla f\|^2}{2}d\nu=\dv_f\left(\nu\nabla f\right)df+fd\lambda.
	\end{align*}
	
Taking exterior derivative, we get 
\begin{align*}
\frac{1}{2}d\|\nabla f\|^2\wedge d\nu=d\big[\dv_f(\nu\nabla f)-\lambda \big]\wedge df,
\end{align*}
and by considering the exterior product with $\ df$ we conclude that
\begin{align*}
d\|\nabla f\|^2\wedge d\nu\wedge df=0.
\end{align*}

\end{proof}

We recall that the traceless tensor associated to a tensor $T$ on $(M^n, g)$ is defined by 
\begin{align*}
	\mathring{T}=T-\frac{\tr_g T}{n}g.
\end{align*}

\begin{lemma}\label{lemeqRic}
Let $(M^n,g,f,\lambda,\nu)$ be a generalized quasi-Einstein manifold such that  $\nu=v\circ f$ for some smooth function $v:I\subset\mathbb{R}\to\mathbb{R},$ where $I=f(M)$ and  let $\phi$ be given by \eqref{OD1}. Then 
\begin{align}\label{gqEm1}
	Ric_x+\frac{1}{d(\phi\circ f)_x}\nabla^2(\phi\circ f)_x=\lambda(x) g_x,
\end{align}
for all $x\in M$.  Moreover, 
\begin{align}\label{tcles}
\mathring{Ric}_x=-\frac{1}{d(\phi\circ f)_x}\mathring{\nabla}^2(\phi\circ f)_x,
\end{align}
for all $x\in M$ and $t=f(x)\in I.$
 In particular, $(M^n,g)$ is an Einstein manifold if, and only if,  $\nabla (\phi\circ f)$ is a conformal vector field. 
\end{lemma}

\begin{proof}
By considering  the hessian of $\phi\circ f,$ for $x\in M$, we have that 
\begin{align*}
\left(\nabla^2 f\right)_x = \frac{1}{\phi'(t)}\nabla^2(\phi\circ f)_x -\frac{\phi''(t)}{\phi'(t)}df_x\otimes df_x, 
\end{align*} 
where $\phi'(t)=d(\phi\circ f)_x\neq0$ and $t=f(x)\in I$. Since $\phi$ is given by \eqref{OD1}, it follows that $\phi''(t)+v(t)\phi'(t)=0$. Therefore, substituting into \eqref{gqEm}, we conclude that \eqref{gqEm1} holds.  

From \eqref{gqEm1}, we can easily check the identity \eqref{tcles}.
Moreover, it also follows from \eqref{gqEm1} that  $(M^n,g)$ is an Einstein manifold if, and only if, $\nabla(\phi\circ f)$ is a conformal vector field. 
\end{proof}

\begin{lemma}\label{lemcon}
Let $(M^n,g)$ be a connected Einstein manifold, $n\geqslant3,$ and let $u:M\to\mathbb{R}$ be a smooth function  whose gradient $\nabla u$ is a conformal vector field. Then 
\begin{align*}
\nabla^2u=\left[-\frac{Su}{n(n-1)}+c\right]g,
\end{align*}
for some constant $c.$
\end{lemma}

\begin{proof}
Since $(M^n,g)$ is an Einstein manifold, i.e., $Ric=\frac{S}{n}g,$ then $S$ is constant and since  $\nabla u$ is a conformal vector field we get 
\begin{align*}
\nabla^2u=\frac{\Delta u}{n}g.
\end{align*}

Recall two well known  facts in the literature 
\begin{align*}
Ric(\nabla u,\cdot)+d\Delta u=\dv(\nabla^2u) \ \ \mbox{and} \ \ \dv(wg)=dw,
\end{align*}
for smooth functions $u$ and $w$.  Taking $w=\frac{\Delta u}{n}$ we get
\begin{align*}
\frac{S}{n}du+d\Delta u=\frac{1}{n}d\Delta u.
\end{align*}	
Thus, dividing by $n-1,$ we have
\begin{align*}
d\left[\frac{Su}{n(n-1)}+\frac{\Delta u}{n}\right]=0.	
\end{align*}
and hence 
\begin{align*}
\frac{\Delta u}{n}=-\frac{Su}{n(n-1)}+c,
\end{align*}
for some constant $c.$
\end{proof}	

We conclude this section by stating Karp's theorem, which is an extension of Stokes theorem and it will  be used in the proof of Theorem \ref{R2}. 

\begin{lemma}[Karp's theorem \cite{karp}]\label{karp} Let $(M^n,g)$ be a complete noncompact Riemannian manifold. Consider the geodesic ball $B(r)$ of radius $r$ centered at some fixed $x\in M^n$ and a vector field $X$ such that
\begin{align*}
\displaystyle{\liminf_{r\to\infty}} \frac{1}{r}\int\limits_{B(2r)\backslash B(r)}\|X\|dvol_g=0.
\end{align*}
If $\dv X$ has an integral (i.e. if  either $(\dv X)^+$ or $(\dv X)^-$ is integrable), then $\int_M\dv Xdvol_g=0$. In particular, if $\dv X$ does not change sign outside some compact set, then $\int_M\dv Xdvol_g=0$. 
\end{lemma}

\section{Proof of the rigidity results}\label{proofrigid}

\subsection{Proof of Theorem \ref{R1}} 
\begin{proof}
We consider $\phi:I\subset\mathbb{R}\to\mathbb{R}$ given by \eqref{OD1}. Then, it follows from Lemma \ref{lemeqRic}  that  \eqref{gqEm} is equivalent to \eqref{gqEm1}. Hence, by computing $\dv\big(\mathring{Ric}(\nabla(\phi\circ f))\big)$ we have
\begin{align*}
\dv\big(\mathring{Ric}(\nabla(\phi\circ f))\big)(x)&=\frac{n-2}{2n}\langle\nabla S_x,\nabla(\phi\circ f)_x\rangle+\langle Ric_x,\mathring{\nabla}^2(\phi\circ f)_x\rangle\\
&=\left\langle \lambda(x) g_x-\frac{1}{d(\phi\circ f)_x}\nabla^2(\phi\circ f)_x,\mathring{\nabla}^2(\phi\circ f)_x\right\rangle\\
&=-\frac{1}{d(\phi\circ f)_x}\|\mathring{\nabla}^2(\phi\circ f)\|^2(x),
\end{align*}
for all $x\in M$. Therefore, since $S$ is constant, we get 
\begin{align}\label{Eq.div}
\dv\big(\mathring{Ric}(\nabla(\phi\circ f))\big)(x)=-\frac{1}{d(\phi\circ f)_x}\|\mathring{\nabla}^2(\phi\circ f)\|^2(x).
\end{align}
By integrating and using the divergence theorem, since $M$ is compact we conclude that  
\begin{align*}
0&=\int_M\dv\big(\mathring{Ric}(\nabla(\phi\circ f))\big)dvol_g=-\int_M\frac{1}{d(\phi\circ f)_x}\|\mathring{\nabla}^2(\phi\circ f)\|^2(x)dvol_g(x),
\end{align*}
which implies that $(M^n,g)$ is an Einstein manifold and $\nabla(\phi\circ f)$ is a conformal vector field.  It follows from Lemma \ref{lemcon} that 
\begin{align*}
\nabla^2(\phi\circ f)=\left[-\frac{S(\phi\circ f)}{n(n-1)}+c\right]g.
\end{align*}	
for some constant $c$. 

Since $M$ is compact,  Obata's theorem \cite{obata} implies  that $(M^n,g)$ is isometric to a standard sphere $\mathbb{S}^n(r)$ and we can write
\begin{align*}
\mathfrak{L}_{\nabla(\phi\circ f)}g=2\psi g,
\end{align*}
where the conformal factor $\psi=\frac{\Delta(\phi\circ f)}{n}$ satisfies the equation $\Delta\psi+\frac{S}{n-1}\psi=0$ (see e.g. Yano~\cite{yano}). 

Rescaling the metric, we can assume that $S = n(n-1)$.  Hence,  $\Delta(\phi\circ f)$ is the first eigenfunction of
the unit  sphere $\mathbb{S}^n$. Therefore, there is a fixed vector $w\in\mathbb{S}^n$ such that $\Delta(\phi\circ f)=h_w=-\frac{\Delta h_w}{n}$ (see \cite{berger}), where $h_w$ is the height function on the unit sphere.   
Therefore, $\phi\circ f=c-\frac{h_w}{n}$ for some constant $c$.  Since $\phi$ is strictly monotone,
%($c_1\neq0$), 
 we conclude that  $f=\phi^{-1}\left(c-\frac{h_w}{n}\right)$.  
\end{proof}

\subsection{Proof of Theorem \ref{R2}}
\begin{proof}
As in the previous proof,  by hypothesis we have 
\begin{align}\label{dvrc}
\dv\big(\mathring{Ric}(\nabla(\phi\circ f))\big)(x)&=-\frac{1}{d(\phi\circ f)_x}\|\mathring{\nabla}^2(\phi\circ f)\|^2(x),
\end{align} 
for all $x\in M$.  The constant $c_1$ of \eqref{OD1} is nonzero,  therefore  $\dv\big(\mathring{Ric}(\nabla(\phi\circ f))\big)$ does not change sign on $M^n$.
 By integrating \eqref{dvrc}, it follows from Lemma \ref{karp} (Karp's theorem) that  
\begin{align*}
-\int_M\frac{1}{\phi'(t)}\|\mathring{\nabla}^2(\phi\circ f)\|^2(x)dvol_g(x)=\int_M\dv\big(\mathring{Ric}(\nabla(\phi\circ f))\big)dvol_g=0.
\end{align*}

Thus, $\nabla(\phi\circ f)$ is a conformal vector field and  hence $(M^n,g)$ is an Einstein manifold. Therefore, Lemma \ref{lemcon} implies  that 
\begin{align}\label{confvector}
\nabla^2(\phi\circ f)=\left(-\frac{S(\phi\circ f)}{n(n-1)}+c\right)g, 
\end{align}
and hence the classification of $(M^n,g)$ depends on the scalar curvature $S$ and on the existence of critical points of $f$. From Bonnet-Myers theorem we have that the constant scalar curvature $S\leq 0$.     

If $S=0$, then either $c=0$ or $c\neq0$.  So, if $c=0,$ then $(M^n,g)$ is isometric to $\mathbb{R}\times N^{n-1}$ (see \cite{kerbrat}) and if $c\neq0$ then $(M^n,g)$ is isometric to the Euclidean space $(\mathbb{R}^n,g_\circ)$ (see \cite{kerbrat} or Theorem $5.5$ in  \cite{B-Tenen}), which proves $(i)$ and $(ii).$  

If $S<0$, then we must consider the critical points of $f$, which are also the same critical points of $\phi\circ f$. If $f$ has any critical point then $(M^n,g)$ is isometric to $\mathbb{H}^n\left(\sqrt{\frac{n(n-1)}{|S|}}\right)$, (see \cite{kanai} or   Theorem $5.7$ in \cite{B-Tenen}), which proves $(iii)$.  On the other hand, if $f$ has no critical points, then $(M^n,g)$ is isometric to $\mathbb{R}\times N^{n-1}$, with the metric given by  $dt^2+e^{2\sqrt{\frac{|S|}{n(n-1)}}t}g_N,$ where $(N^{n-1},g_N)$ is a Riemannian Einstein manifold, (see \cite{kerbrat} or  Theorem $5.7$ in  \cite{B-Tenen}). This proves $(iv).$ 
\end{proof}

\subsection{Proof of Corollary \ref{aaa}}
\begin{proof}
Initially, we note that $(M^n,g)$ has constant scalar curvature $S$ and there is a conformal Einstein metric $\overline{g}$,  whose Ricci tensor $\overline{Ric}$ satisfies
\begin{align*}
\overline{Ric}=Ric+\frac{(n-2)}{\varphi}\nabla^2\varphi+\frac{\Delta\varphi}{\varphi}g-(n-1)\frac{\|\nabla\varphi\|^2}{\varphi^2}g.	
\end{align*}
Moreover, 
\begin{align*}
	\overline{Ric}=\frac{\overline{S}}{n}\overline{g}=\frac{\overline{S}}{n\varphi^2}g,
\end{align*}
where $\overline{S}$ is the (constant) scalar curvature of $\overline{g}.$ 

So,
\begin{align}\label{var}
Ric+\frac{(n-2)}{\varphi}\nabla^2\varphi=\lambda g,	
\end{align}
with 
\begin{align*}
\lambda=\frac{\overline{S}}{n\varphi^2}-\frac{\Delta\varphi}{\varphi}+(n-1)\frac{\|\nabla\varphi\|^2}{\varphi^2}.	
\end{align*}

Now, taking $f=(n-2)\ln\varphi,$ i.e., $\varphi=\phi\circ f,$ by a straightforward calculation $\left(M^n,g,f,\lambda,-\frac{1}{n-2}\right)$ is a generalized quasi-Einstein manifold and since
\begin{align*}
\displaystyle{\liminf_{r\to\infty}} \frac{1}{r}\int\limits_{B(2r)\backslash B(r)}\|\mathring{Ric}(\nabla(\phi\circ f))\|dvol_g&=\displaystyle{\liminf_{r\to\infty}} \frac{1}{r}\int\limits_{B(2r)\backslash B(r)}\|\mathring{Ric}(\nabla\varphi)\|dvol_g=0,
\end{align*}
the classification of $(M^n,g)$ follows by Theorem \ref{R2}.
\end{proof}

\section{ Classes of Generalized Quasi-Einstein Manifolds}\label{nonconstantS}

In this section, we prove Theorems \ref{radialcomplete} and \ref{translation},  which show the existence of  classes of  complete  generalized quasi-Einstein manifold, that  are conformal to the Euclidean 
space. We also provide explicit examples.  
We observe that the manifolds have the property that besides $\nu$ being  a function of the potential  $f$, as in the previous section, the functions $\varphi$ and also $\lambda$ are functions of $f$. 
 
Let $(\mathbb{R}^n,g_\circ)$ be the Euclidean space, with  coordinates $x=(x_1,\ldots,x_n)$, $n\geqslant 3$,  where $g_\circ$  is the standard metric. In our next lemma, we present the differential equations that must be satisfied by the functions of a  generalized quasi-Einstein manifold, which are  conformal to the Euclidean space.  

\begin{lemma}\label{prop4}
Let $(\mathbb{R}^n,g_\circ),\ n\geqslant3,$ be the Euclidean space. Consider smooth functions $\varphi, \ f,\ \lambda$ and $\nu:\mathbb{R}^n\to\mathbb{R}$,  such that $\varphi$ does not vanish. Then $(\mathbb{R}^n,g,f,\lambda,\nu),$ with $g=g_\circ/{\varphi^2}$ is a nontrivial generalized quasi-Einstein manifold if, and only if, the functions $f,\lambda,\nu$ and $\varphi$ satisfy 
for all $i\neq j$, $1\leq i,j\leq n$,   
\begin{align}\label{arad1}
(n-2)\frac{\varphi_{x_ix_j}}{\varphi}+f_{x_ix_j}+f_{x_i}\frac{\varphi_{x_j}}{\varphi}+\frac{\varphi_{x_i}}{\varphi}f_{x_j}-\nu f_{x_i}f_{x_j}=0
\end{align} 	
and for all $i$, 
\begin{align}\label{arad2}
\nonumber&(n-2)\frac{\varphi_{x_ix_i}}{\varphi}+f_{x_ix_i}+2f_{x_i}\frac{\varphi_{x_i}}{\varphi}-\nu\cdot(f_{x_i})^2+\\
&+\sum_{k=1}^n\left[\frac{\varphi_{x_kx_k}}{\varphi}-(n-1)\left(\frac{\varphi_{x_k}}{\varphi}\right)^2-f_{x_k}\frac{\varphi_{x_k}}{\varphi}\right]=\frac{\lambda}{\varphi^2}.
\end{align}
\end{lemma}

\begin{proof}
The proof follows by  computing each term of the fundamental equation \eqref{gqEm} for the  metric $g_{ij}=\delta_{ij}/{\varphi^2}$. 
Since 
\begin{align*}
(Ric_{g})_{ij}=\frac{(n-2)}{\varphi}\varphi_{x_ix_j}+\sum_{k=1}^n\left[\frac{\varphi_{x_kx_k}}{\varphi}-(n-1)\left(\frac{\varphi_{x_k}}{\varphi}\right)^2\right]\delta_{ij},
\end{align*}
\begin{align*}
(\nabla^2f)_{ij}=f_{x_ix_j}+f_{x_i}\frac{\varphi_{x_j}}{\varphi}+\frac{\varphi_{x_i}}{\varphi}f_{x_j}-\frac{1}{\varphi}
\sum_{k=1}^n f_{x_k}\varphi_{x_k}\delta_{ij},
\end{align*}
$(df\otimes df)_{ij}=f_{x_i}f_{x_j}$, we conclude that 
\eqref{gqEm} holds if, and only if,  \eqref{arad1} and \eqref{arad2} hold.

\end{proof}	

Observe that the scalar curvature of the metric $g=g_0/\varphi^2$ is given by 
\begin{equation}\label{Sconformal}
S= (n-1)(2\varphi\Delta_{g_0}\varphi-n|\nabla_{g_0}\varphi|^2).
\end{equation} 

We point out that the system of equations \eqref{arad1} and \eqref{arad2} admits radial solutions and also solutions that are invariant under the action of an $(n-1)$ dimensional  translation group. These  solutions provide the generalized quasi-Einstein manifolds given in Theorems \ref{radialcomplete} and \ref{translation}.  

\subsection{Proof of Theorem \ref{radialcomplete}. Examples } 
\begin{proof} 
Assume that $\varphi$ and $f$ are functions of $r=\|x\|^2$. 
 As a consequence of Lemma 4.1,  $(\R^n,g,f,\lambda,\nu)$ is a generalized quasi-Einstein manifold, where $g=g_0/\varphi^2$ if, and only if,
 \eqref{arad1} and \eqref{arad2} hold.  
 A straightforward computation shows that these equations reduce, respectively, to
\begin{align}\label{aradcomp1}
 4x_ix_j\left[(n-2)\frac{\varphi''}{\varphi}+f''+2f'\frac{\varphi'}{\varphi}
-\nu(f')^2 \right]=0,
\end{align}
 for all $i\neq j$, $x\in \mathbb{R}^n$ and for all $i$, 
\begin{align}\label{aradcomp2}
& 4x_ i^2\left[(n-2)\frac{\varphi''}{\varphi}+f''+2f'\frac{\varphi'}{\varphi}
-\nu(f')^2 \right]+ 2f'\nonumber \\
& +4\left[(n-1)\frac{\varphi'}{\varphi} +\frac{\varphi''}{\varphi} r -(n-1)
\left(\frac{\varphi'}{\varphi} \right)^2 r- f'\frac{\varphi'}{\varphi}r \right]=\frac{\lambda}{\varphi^2}.
\end{align}
 
Since $f(r)$ is a strictly monotone function, then $f'\neq 0$. Therefore, it follows from  \eqref{aradcomp1}  and \eqref{aradcomp2} that 
  $(\mathbb{R}^n,g,f,\lambda,\nu)$  is a generalized quasi-Einstein manifold if, and only if,  $\nu$  is given in terms of 
 $\varphi$ and $f$ by  \eqref{rad1n} and  $\lambda$ is defined by \eqref{rad2n}. The expression \eqref{radS}  of the scalar curvature follows from \eqref{Sconformal}.
 
 If  $\varphi$ is a bounded function, then the  completeness of the manifold follows from the fact that the length of any divergent curve is unbounded.   

\end{proof}

\vspace{.2in}

We now present a couple of examples of complete, conformally flat, generalized quasi-Einstein manifold in  $\mathbb{R}^n$,   described in Theorem \ref{radialcomplete}. 

\vspace{.1in} 

\noindent {\bf Example 1.} We consider the conformal factor $\varphi(r)$, where $r=\|x\|^2$, given by 
 $ \; 
\varphi(r)=e^{-\frac{r^2}{2}}   
%\varphi(r)=a+ e^{-\frac{r^b}{b}},\qquad a\geq 0, \quad %b\in\mathbb{N}\setminus\{1\}, 
$
and the potential  of the gaussian soliton,  
\;  $f(r)=c r$,  \; $c\neq 0$.
From equations \eqref{rad1n} and \eqref{rad2n} we obtain
\[
\nu(r)= \frac{1}{c^2}(nr^2-2cr-2r^2-n+2)
%\nu=\frac{r^{b-2}}{c^2(a e^{\frac{r^b}{b}}+1)}\left[ -2cr+(n-2)(r^b-b+1) %\right]
\]
and 
\[
\lambda(r)=4e^{-r^2}r(-nr^2+cr+2r^2-n)+2ce^{-r^2}.
%\lambda=-4r^{b-1}e^{- \frac{r^b}{b}}\left( (e^{- \frac{r^b}{b}}+a)\ [ (n-2)%(r^b+1)+b-cr]   -a(n-1)r^b\right).
\]
Since $\| \varphi \|< 1$,  it follows from Theorem \ref{radialcomplete} that $(\mathbb{R}^n, g_\circ/\varphi^2, f,\lambda,\nu)$   is a complete generalized quasi-Einstein manifold, whose non constant (negative) scalar curvature is given by   
\[
S(r)= -4(n-1)e^{-r^2}r(nr^2-2r^2+n+2).
%S= 4(n-1)r^{b-1}e^{-\frac{r^b}{b}}\left( (a+e^{-\frac{r^b}{b}})\ 
%[2(r^b-b+1)-n]+nr \right). 
\] 
\vspace{.1in} 

\noindent {\bf Example 2.} We consider the conformal factor $\varphi(r)=1/(1+r)   
$, where $r=\|x\|^2$ and the potential  of the gaussian soliton,  
\;  $f(r)=c r$,  \; $c\neq 0$.
From equations \eqref{rad1n} and \eqref{rad2n} we obtain
\[
\nu(r)= \frac{2}{c^2(1+r)^2} [n-2-c(1+r)]
\]
and
\[
\lambda(r)=\frac{4}{(1+r)^4}\left[cr(1+r)-2r(n-2)-n+1)  \right] +\frac{2c}{(1+r)^2}. 
\]
Since $\| \varphi \|< 1$, Theorem \ref{radialcomplete} implies  that $(\mathbb{R}^n, g_\circ/\varphi^2, f,\lambda,\nu)$   is a complete generalized quasi-Einstein manifold, whose  scalar curvature is given by   
\[
S(r)=\frac{4(n-1)}{(1+r)^4}(4r-2nr-n).
\]

\vspace{.1in}

One can certainly produce many other explicit examples of complete 
generalized quasi-Einstein manifolds, by choosing  other  radial  potential functions $f(r)$ that are strictly monotone or other bounded and non vanishing radial functions $\varphi(r)$, where $r=\|x\|^2$.

\subsection{Proof of Theorem \ref{translation}. Examples } 

\begin{proof}    
Assume that $f(u)$ and $\varphi(u)$ are functions of $u=\sum_{k=1}^n \alpha_k x_k$, where $a=\sum_{k=1}^n\alpha_k^2\neq 0$,  such that $\varphi$ does not vanish. As a consequence of Lemma 4.1,  $(\R^n,g,f,\lambda,\nu)$ is a generalized quasi-Einstein manifold, where $g=g_0/\varphi^2$ if, and only if, for all $i\neq j$ 
\begin{equation}\label{transij}
\alpha_i\alpha_j\left[(n-2)\frac{\varphi''}{\varphi}+f''+2f'\frac{\varphi'}{\varphi} - \nu (f')^2\right]=0,
\end{equation}
and for all $i$  
\begin{eqnarray} 
\lambda&=& \alpha_i^2[(n-2)\varphi\varphi''+\varphi^2f''+2\varphi f'\varphi '-
\nu\varphi^2(f')^2]+\nonumber\\
 &&+ a \,[\varphi\varphi''-(n-1)(\varphi')^2-f'\varphi\varphi'].\label{transii}
\end{eqnarray}

If there exists $i\neq j$ such that $\alpha_i\alpha_j\neq 0$, then it follows from \eqref{transij} that
\begin{equation}\label{eqnutransl}
(n-2)\frac{\varphi''}{\varphi}+f''+2f'\frac{\varphi'}{\varphi} - \nu (f')^2=0. 
\end{equation}
Hence \eqref{transii} reduces to 
\begin{equation}\label{lamtransldemo}
\lambda=a \,[\varphi\varphi''-(n-1)(\varphi')^2-f'\varphi\varphi'].
\end{equation}

If for all pairs  $i \neq j $, $\alpha_i\alpha_j=0$, then $u$ depends only on one of the variables $x_i$ and  without loss of generality, we may assume $u=\alpha_1 x_1$, $\alpha_1\neq 0$.
Then \eqref{transij} is trivially satisfied for all $i\neq j$.  Moreover,  for $i>1$, \eqref{transii}  reduces to \eqref{lamtransldemo} with 
$a=\alpha_1^2\neq 0$
and for $i=1$, it reduces to 
\begin{eqnarray*} 
\lambda= \alpha_1^2[(n-2)\varphi\varphi''+\varphi^2f''+2\varphi f'\varphi '-
\nu\varphi^2(f')^2]+\\
+a \,[(n-2)\frac{\varphi''}{\varphi}+f''+2f'\frac{\varphi'}{\varphi} - \nu (f')^2],  
\end{eqnarray*}
Therefore, since  $\lambda$ satisfies \eqref{lamtransldemo}, it follows from  
the last  equation that \eqref{eqnutransl} holds. 

Hence, we conclude that in both cases the functions $\nu$ and  $\lambda$ are determined in terms of $\varphi$ and $f$ by \eqref{nutransl} and \eqref{lambdatransl} respectively.

The expression of the scalar curvature follows from \eqref{Sconformal} and is given by \eqref{Stransl}. If  $\varphi$ is a bounded function, then the  completeness of the manifold follows from the fact that the length of any divergent curve is unbounded.   

\end{proof}

{\bf Example 3.} We consider the conformal factor $\varphi(u)=1+\tanh u$,    
 where $u=\sum_{k=1}^n\alpha_k x_k$, with $a=\sum_{k=1}^n\alpha_k^2\neq 0$ and the potential function   
\;  $f(u)= u$. From equations  \eqref{nutransl} and \eqref{lambdatransl}, we obtain
\[
\nu=\frac{2\,[1-(n-2)\tanh u]}
{(\cosh u)^2\,(1+\tanh u)}   
\]
 and 
 \[
 \lambda= \frac{a}{(\cosh u)^2}[(n-3)(\tanh u)^2 -3\tanh u -n] .
 \]
 Since $\varphi$ is a bounded function it follows from Theorem \ref{translation} that  $(\mathbb{R}^n, g_\circ/\varphi^2, f,\lambda,\nu)$   is a complete generalized quasi-Einstein manifold, whose  scalar curvature is  given by  
 \[
 S=\frac{a(n-1)}{(\cosh u)^2}[(n-4)(\tanh u)^2-4\tanh u -n].
 \]

\vspace{.1in}

One can produce many other explicit examples of complete 
generalized quasi-Einstein manifolds, by choosing other  strictly monotone potential functions $f(u)$ or  bounded and non vanishing functions $\varphi(u)$, where $u=\sum_{k=1}^n\alpha_k x_k$ and $a=\sum_{k=1}^n\alpha_k^2\neq 0$, which are invariant under the action of an $(n-1)$ dimensional translation group.

\end{document}